\documentclass[a4paper,11pt,reqno]{amsart}
\usepackage[utf8]{inputenc}
\usepackage{amsmath,amsthm,amssymb,amsfonts,extarrows,mathrsfs,comment}
\usepackage{stmaryrd}
\usepackage{enumitem}
\setlist[enumerate]{label=\rm{(\roman*)},leftmargin=\parindent,itemindent=\parindent,labelsep=5pt}
\newlist{TFAE}{enumerate}{1}
\setlist[TFAE]{label=\rm{(\alph*)},labelindent=2\parindent}
\usepackage{tikz-cd}
\usepackage{mathtools}
\usepackage{xcolor}
\usepackage[margin=1in]{geometry}
\usepackage[colorlinks=true,citecolor=magenta,linkcolor=blue]{hyperref}

\title[A note on FM partners of abelian varieties over positive characteristic fields]{A note on Fourier-Mukai partners of abelian varieties over positive characteristic fields}

\author{Zhiyuan Li}
\address{Shanghai Center for Mathematical Science \\2005 Songhu Road 200438, Shanghai}
\email{zhiyuan\_li@fudan.edu.cn}

\author{Haitao Zou}
\address{Shanghai Center for Mathematical Science \\2005 Songhu Road 200438, Shanghai}
\email{htzou17@fudan.edu.cn}

\subjclass[2010]{Primary 14F05; Secondary 14K05}
\keywords{Fourier-Mukai partner, abelian variety, derived Torelli Theorem, Kummer variety}

\newcommand\cA{\mathcal{A}}

\newcommand\cC{\mathcal{C}}

\newcommand\cE{\mathcal{E}}
\newcommand\cF{\mathcal{F}}

\newcommand\cK{\mathcal{K}}
\newcommand\cL{\mathcal{L}}

\newcommand\cN{\mathcal{N}}
\newcommand\cO{\mathcal{O}}
\newcommand\cP{\mathcal{P}}
\newcommand\cQ{\mathcal{Q}}

\newcommand\cT{\mathcal{T}}

\newcommand\cX{\mathcal{X}}
\newcommand\cY{\mathcal{Y}}
\newcommand\cZ{\mathcal{Z}}

\newcommand\rH{\mathrm{H}}

\newcommand{\loccit}{{\it loc.cit.~}}
\newcommand{\cf}{{\it cf.~}}
\DeclareMathOperator{\D}{D}

\DeclareMathOperator{\Frac}{Frac}
\DeclareMathOperator{\spec}{Spec}

\DeclareMathOperator{\dR}{dR}

\DeclareMathOperator{\id}{id}
\DeclareMathOperator{\Char}{char}

\DeclareMathOperator{\Hom}{Hom}
\DeclareMathOperator{\coh}{Coh}
\DeclareMathOperator{\Qcoh}{QCoh}

\DeclareMathOperator{\km}{\rm Km}
\DeclareMathOperator{\rK}{\rm K}

\DeclareMathOperator{\NS}{NS}
\DeclareMathOperator{\Pic}{Pic}
\DeclareMathOperator{\Z}{\mathbb{Z}}

\newcommand{\ZZ}{\mathbb{Z}}

\DeclareMathOperator{\et}{\text{\'et}}

\DeclareMathOperator{\im}{Im}

\DeclareMathOperator{\crys}{crys}

\DeclareMathOperator{\tr}{tr}

\DeclareMathOperator{\hilb}{Hilb}

\DeclareMathOperator{\Aut}{Aut}
\DeclareMathOperator{\Cong}{\cong}

\newcommand{\bK}{\mathbb{K}}
\numberwithin{equation}{subsection}
\newcounter{ini}
\newtheorem{thm}{Theorem}[section]
\newtheorem{lemma}[thm]{Lemma}
\newtheorem{prop}[thm]{Proposition}
\newtheorem{cor}[thm]{Corollary}

\newtheorem{question}[thm]{Question}

\theoremstyle{remark}
\newtheorem{rmk}[thm]{Remark}
\theoremstyle{definition}
\newtheorem{defn}[thm]{Definition}
\theoremstyle{remark}

\begin{document}
\begin{abstract}
Over complex numbers, the Fourier-Mukai partners of abelian varieties are well-understood. A celebrated result is Orlov's derived Torelli theorem. In this note, we study the FM-partners of abelian varieties in positive characteristic. We notice that, in odd characteristics, two abelian varieties of odd dimension are derived equivalent if their associated Kummer stacks are derived equivalent, which is Krug and Sosna's result over complex numbers. For abelian surfaces in odd characteristic, we show that two abelian surfaces are derived equivalent if and only if their associated Kummer surfaces are isomorphic. This extends the result \cite{HosonoLianEtAl2003} to odd characteristic fields, which solved a classical problem originally from Shioda.  Furthermore, we establish the derived Torelli theorem for supersingular abelian varieties and apply it to characterize the quasi-liftable birational models of supersingular generalized Kummer varieties. 
\end{abstract}
\maketitle

\section{Introduction}

Let $A$ be an abelian variety over an algebraically closed field $k$. It is desirable to have a description of the derived category of $A$ via its associated Kummer stack $\mathrm{K}(A)\coloneqq [A/\iota]$, where $\iota$ acts on $A$ by the involution $-1 \colon A \to A$. When $A$ is an abelian surface, the singular Kummer variety $A/\iota$ admits a crepant resolution, denoted by $\km(A)$.
A classical problem raised by Shioda is 
\begin{question}[\cite{Shioda1977}]\label{Q1}
{For two abelian surfaces $A_1$ and $A_2$, if $\km(A_1) \cong \km(A_2)$, then can we conclude that $A_1 \cong A_2$?}
\end{question}
Over complex numbers, Question \ref{Q1} was solved in \cite{HosonoLianEtAl2003,orlov97}: two abelian surfaces have isomorphic Kummer surfaces if and only if they are derived equivalent.  Their proof relies on the use of Hodge theory and global Torelli theorem for abelian surfaces, which are missing in positive characteristic fields.  More generally, Stellari has investigated this problem for abelian varieties of arbitrary dimension in \cite{Stellari2007}.

In this paper, we are interested in above questions over positive characteristic fields. In particular, we would like to extend the result of \cite{HosonoLianEtAl2003} and \cite{Stellari2007} to  fields with odd characteristic. The following result gives an answer of  Shioda's question over positive characteristic fields. 
\begin{thm}\label{mainthm1}
Assume $\mathrm{char}(k)\neq 2$. Let $A_1$ and $A_2$ be abelian varieties of dimension $n$ over $k$. Then the following holds:
\begin{enumerate}
    \item If $A_1$ is derived equivalent to $A_2$, then Kummer stack $\rK(A_1)$ is derived equivalent to $\rK(A_2)$.
    \item If $n$ is odd, then the converse of (i) holds.
    \item If $n=2$, then $A_1$ and $A_2$ are derived equivalent if and only if $\km(A_1)\cong \km(A_2)$. If $A_1$ is supersingular, then $A_2$ is derived equivalent to $A_1$ if and only if $A_2\cong A_1$. 
\end{enumerate}
\end{thm}

The statements (i) and (ii) will be proved with techniques around equivariant derived categories, which is already known in  \cite{KrugSosna2015,ploog07} with $k=\mathbb{C}$. In the same way, we will generalize them to the case that $\Char(k)>2$ . 

The proof of statement (iii) will be divided into two cases.
\begin{itemize}
    \item For the finite height case, we can prove a lifting theorem for Kummer structures (see \S \ref{subsec:liftingKummerstructure}). Then the specialization argument of derived equivalences will imply this statement.
    \item For the supersingular case, it can be concluded by supersingular Torelli theorem for abelian varieties (see \S \ref{subsec:supersingularderivedtorelli}).
\end{itemize}
Notice that Theorem \ref{mainthm1} shows that supersingular abelian surfaces do not have any non-trivial Fourier-Mukai partners.  We expect that there is a similar characterization for higher dimension supersingular abelian varieties. 

 As an application, we  can characterize the quasi-liftably birational class (\cf\cite[Definition 3.3]{FL18}) of irreducible symplectic varieties which come from the moduli space of sheaves on supersingular abelian surfaces. Consider the moduli space of Gieseker-Maruyama $H$-stable sheaves on $A$ with Mukai vector $v$ such that $p\nmid v^2$, denoted by $M_H(A,v)$. Recall that the fiber $K_v(A)$  of the Albanese morphism 
 \[
 (\det, c_{2})\colon M_{H}(A,v)\to \Pic^{0}(A)\times A,
 \]
 is an irreducible symplectic variety of dimension $v^2-2$ in the sense of \loccit. A consequence in \loccit asserts that the generalized Kummer type variety $K_v(A)$ is quasi-liftably birational to some generalized Kummer variety $K_n(A')$, where $A'$ is a Fourier-Mukai partner of $A$. Here we verify that $A \cong A'$. 
\begin{thm}\label{thm:supersingulargeneralizedKummer}Let $n = \frac{v^2}{2}$. Suppose that $p \nmid n$ and $p \nmid n+1$. Let $A$ be a supersingular abelian surface and let $K_v(A)$ be the generalized Kummer type variety. Then 
\begin{enumerate}
    \item $K_v(A)$ is quasi-liftably birational equivalent to $K_n(A)$. 
    \item if $K_v(A)$ is quasi-liftably birational to $K_v(A')$ for some abelian surface $A'$ and $v'$, then $A\cong A'$. 
\end{enumerate}
\end{thm}

\noindent\textbf{Acknowledgement:} We are grateful to Lie Fu for helpful discussions and comments. We also want to thank the referee for several useful comments. The authors  are  supported by NKRD Program of China (No. 2020YFA0713200), NSFC General Program (No.~11771086) and Shanghai Pilot Program for Basic Research (No. 21TQ00).

\section{Equivariant derived category of schemes}
\label{equi-dc}
\subsection{Equivariant quasi-coherent sheaves} \label{sec:eqcoh}
Let $G$ be a constant finite group scheme over $k$ and let $X$ be a quasi-compact and quasi-separated scheme over a field $k$ with a $G$-action $$\mu \colon G \times_k X \to X.$$
A {\em $G$-equivariant quasi-coherent sheaf} on $X$ is a pair $(\cF,\lambda)$ such that $\cF$ is a quasi-coherent sheaf on $X$ and $\lambda$ is a family $ \left\{\lambda_g \colon \cF \to g^*\cF \right\}_{g \in G}$ of isomorphisms satisfying the following cocycle condition:
    \begin{equation}
    \label{eq:cocyclecondition}
   \lambda_{fg} = \lambda_f \circ \lambda_g \colon \cF \to g^* \cF \to (fg)^*\cF.
    \end{equation}
We call $\lambda$ a \emph{$G$-linearization} of $\cF$. For instance, $\cO_{X}^{\rm can}= (\cO_X, \id_{\cO_{X \times G} })$ is a $G$-equivariant quasi-coherent sheaf. In this paper,  we denote by $\Qcoh_G(X)$ the category of $G$-equivariant quasi-coherent $\cO_X$-modules. 
\begin{rmk}
If the cocyle condition \eqref{eq:cocyclecondition} for $\gamma$ is missing, then $\cF$ will be called \emph{$G$-invariant}.
\end{rmk}

Let $[X/G]$ be the quotient stack given by the $G$-action on $X$ and denote by $\Qcoh([X/G])$ the category of quasi-coherent sheaves on $[X/G]$ (\cf\cite[\S 9]{Olsson2016}). There is a well-known stacky description for $G$-equivariant quasi-coherent sheaves  on $X$ as follows. 
\begin{lemma}
\label{lemma:stackydescrip}
 There is a canonical equivalence of categories \[
 \Qcoh([X/G]) \simeq \Qcoh_{G}(X).
 \]
 Moreover, if $X$ is locally noetherian, then we also have $\coh([X/G]) \cong \coh_G(X)$.
\end{lemma}
\begin{proof}
Assume that $X$ is locally noetherian.
Consider the functor
\[
\begin{aligned}
    \Phi \colon \Qcoh_G(X) &\to \Qcoh([X/G])\\
    (\cF,\lambda) &\mapsto \widetilde{\cF}
\end{aligned}
\]
where $\widetilde{\cF}$ is the sheaf on $[X/G]$ defined as follows.
For any object
\[
\begin{tikzcd}
\mathscr{P} \ar[r,"\pi"] \ar[d] & X_T \\
T
\end{tikzcd}
\]
lying on a $k$-scheme $T$, the pull-back $\pi^* \cF_T$ is a $G$-equivariant quasi-coherent sheaf on $\mathscr{P}$ as $\pi$ is $G$-equivariant. 
The descent theory along $G$-torsor for the stack $\Qcoh_k$ of quasi-coherent sheaves  establishes a canonical equivalence:
\begin{equation}\label{eq:descentalongtorsor}
\Qcoh(T) \simeq \Qcoh_G(\mathscr{P}), 
\end{equation}
see \cite[Theorem 4.46]{vistolinote} for example. Take $\widetilde{F}(T, \mathscr{P}, \pi)$ to be one in the isomorphism class in $\Qcoh(T)$ corresponding to $\pi^* \cF_T$.
It remains to show that $\widetilde{F}$ is quasi-coherent (resp.\ coherent).

Consider the smooth covering $(G_X,\mu) \colon X \to [X/G]$ where $G_X$ is the trivial torsor on $X$ defined by $p_X \colon X \times_k G \to X$ and $\mu$ is the group action of $G$ on $X$. Since the $\lambda$ is an isomorphism
\[
p_X^* \cF \xrightarrow{\sim} \mu^* \cF,
\]
we can see $\widetilde{\cF}(X,G_X,\mu) \cong \cF$
by the previous construction, which is quasi-coherent (resp.\ coherent). Thus by \cite[Proposition 9.1.15]{Olsson2016}, we can see $\widetilde{\cF}$ is quasi-coherent (resp.~coherent).

The converse is similar. We just take $\Phi^{-1} \widetilde{\cF}$ to be the quasi-coherent sheaf (resp.\ coherent sheaf) $\widetilde{\cF}(X,G_X,\mu)$. The linearization $\lambda$ is from the definiton of quasi-coherent sheaves (resp.\ coherent sheaves) on $[X/G]$.
\end{proof}
The Lemma \ref{lemma:stackydescrip} implies that the category $\Qcoh_G(X)$ is a Grothendieck category (\cf\cite[\href{https://stacks.math.columbia.edu/tag/0781}{Tag 0781}]{stacks-project}). This promises a nice homological algebraic theory on $G$-equivariant quasi-coherent sheaves. In the following literature, the \emph{$G$-equivariant derived category} of $X$ means the bounded derived category of $\coh_{G}(X)$,  denoted by $\D^b_G(X)$. A useful fact for $G$-equivariant derived category is the derived McKay correspondence established by Bridgeland--King--Reid \cite{BKR}. 
The restriction of the Hilbert--Chow morphism $X^{[n]} \to X^{(n)}$ gives a morphism
\[
\tau\colon X\sslash G \to X/G.
\]
under the natural inclusion $X\sslash G \hookrightarrow X^{[n]}$.

\begin{thm}[Bridgeland--King--Reid]
\label{thm:BKR}Assume that $(|G|,p)=1$.
Suppose the following two conditions hold
\begin{enumerate}[label=\tt{(BKR\arabic*)},leftmargin=3pc]
    \item $\omega_X$ is locally trivial as a $G$-bundle,
    \item $X\sslash G \times_{\tau} X\sslash G$ has dimension $d \leq \dim X+1$.
\end{enumerate}
 Then there is a derived equivalence between $\D^b(X\!\sslash\!G)$ and $\D^b_G(X)$.
\end{thm}

\begin{proof}
The original proof in \loccit is for $k =\mathbb{C}$. However, this also proceeds for general case that $(|G|,p) =1$ (\cf\cite[Theorem 2.4.5]{dolgachevnote}).
\end{proof}

The following consequence is well-known over complex numbers (\cf\cite[\S 10.2]{BKR} or \cite[\S 3.2]{ploog07}).
\begin{cor}\label{cor:equivariantabelian}
Suppose $\Char(k)\neq 2$. Let $A$ be an abelian surface over $k$. Let $\iota$ be the involution on $A$ and $\langle \iota \rangle$ the finite group scheme over $k$ generated by $\iota$. Then $\D^b([A/\langle\iota\rangle])\simeq \D^b(\km(A))$. 
\end{cor}

\begin{proof} 
We can view $\km(A)$ as the closure of the subset of reduced $\langle\iota\rangle$-clusters in the Hilbert scheme $\hilb^{2}(A)$ of two points on $A$. In this case {\texttt{(BKR2)}} is satisfied and the canonical sheaf is trivial as an $\langle \iota \rangle$-bundle. Thus Theorem \ref{thm:BKR} implies that there is a Fourier-Mukai transform
\[
    \Phi^{P} \colon \D^b(\km(A)) \to \D^b_{\langle \iota \rangle}(A) \simeq \D^b([A/\langle\iota\rangle])
\]
as $\Char(k) \neq 2$.
\end{proof}

Let us recall some general duality theorem for the $G$-equivariant derived category.
The $G$-action on $X$ also induces $G$-action on the triangulated category $\D^b(X)$. Thus we can also consider the $G$-equivariant category $\D^b(X)^{G}$ (\cf\cite[\S 2]{Elagin2014}). Under the assumption that $(|G|,p)=1$, we have an exact equivalence
\begin{equation}\label{eq:equicat}
    \D^b(X)^{G} \simeq \D^b_{G}(X)
\end{equation}
by \loccit Theorem 7.1. Therefore, we will not distinguish $\D^b_{G}(X)$ and $\D^b(X)^{G}$ in the rest of the paper if $G$ acts on $X$ and $(\lvert G \rvert , p)=1$.

Let $\widehat{G}= \Hom(G, k^*)$ be the character group of $G$. There exists an induced $\widehat{G}$-action on the $G$-equivariant derived category $\D^b_G(X)$ as follows. For each $\chi \in \widehat{G}$, one can define a line bundle on $[X/G]$ twisted by $\chi$ as
\[
\cL_{\chi} \coloneqq \cO_{X}^{can} \otimes_k \chi\cong(\cO_X, \id_{\cO_{G\times X}} \otimes \chi) \quad \text{ for $\chi \in \widehat{G}$ }.
\]
The action of $\chi$ on $\D^b([X/G])$ is given by tensoring $\cL_{\chi}$. 
\begin{defn}\label{def:dualaction}
For any $G$-equivariant object $(E,\lambda)$ in $\D^b_G(X)$, the $\widehat{G}$-action on a  $(E,\lambda)$ is given by twisting the linearization:
\[
\lambda \otimes_k \chi \coloneqq \left\{ E \xrightarrow{\lambda_g \cdot \chi(g)} g^* E \right\}_{g \in G} \quad \text{ for any $\chi \in \widehat{G}$}.
\]
This gives a $\widehat{G}$-action on the triangulated categories $\D^b_G(X) \simeq \D^b(X)^G$.
\end{defn}

Consider the identification $\D^b([X/G]) \simeq \D^b(X)^{G}$.
\begin{prop}[A. Elagin]\label{prop:dualities}
Assume $X$ is noetherian. There are exact equivalences
 \[ \D^b_G(X)^{\widehat{G}} \simeq (\D^b(X)^{G})^{\widehat{G}} \simeq \D^b(X).   
\]
\end{prop}
\begin{proof}
The first exact equivalence is given by \eqref{eq:equicat}. For the second exact equivalence, we shall note that $\D^b_G(X)$ is idempotent complete for any noetherian scheme or algebraic stack $X$ as $\coh(X)$ is a Grothendieck category and  admits all products. Now we can apply the duality theorem \cite[Theorem 4.2]{Elagin2014} to conclude it.
\end{proof}

\begin{rmk}
The Proposition \ref{prop:dualities} is also referred for idempotent complete $\mathbb{C}$-linear triangulated categories in \cite{KrugSosna2015} as a crucial fact (\cf Proposition 2.2 \loccit). We use the slightly general form for field $k$ with $(\lvert G \rvert,\Char(k)) =1$.
\end{rmk}

With the same notations in Corollary \ref{cor:equivariantabelian},  we have $\widehat{\langle \iota \rangle} \cong \langle \iota^* \rangle$, where $\iota^*$ is the character dual to $\iota$, acting naturally on $\D^b([A/\langle \iota \rangle])$.
\begin{cor}\label{cor:dualKummerstack}
If $\Char(k) >2$, then $\D^b\left([A/\iota]\right)^{\langle \iota^* \rangle} \simeq \D^b(A)$. In particular, if $A$ is an abelian surface, then $\D^b(\km(A))^{\langle \iota^* \rangle} \simeq \D^b(A)$. \qed
\end{cor}

\section{Lifting of derived equivalences}\label{sec:descend}

\subsection{Equivariant derived equivalences}
In this part, we will recollect some preliminary facts on lifting theory and descent theory of equivariant equivalences in \cite{KrugSosna2015,ploog07} and extend them to all algebraically closed fields. With the notations as in \S 2,  we will always assume the order of $\lvert G \rvert$ is  coprime to $p$.  Let $X_1$ and $X_2$ be two projective $k$-schemes or quotient stacks equipped with $G$ actions. Let $\Phi \colon \D^b(X_1) \to \D^b(X_2)$ be an exact functor.

\begin{defn}\label{defn:liftdescent}
 An exact functor $\widetilde{\Phi} \colon \D^b_G(X_1) \to \D^b_G(X_2)$ is called a {\it descent} of $\Phi$ if it fits into the following 2-commutative diagrams
    \[ \begin{tikzcd}
    \D^b(X_1) \ar[r,"\Phi"] \ar[d,"\pi_{1,*}"] & \D^b(X_2) \ar[d,"\pi_{2,*}"] \\
    \D^b_G(X_1) \ar[r, "\widetilde{\Phi}"] & \D^b_G(X_2)
    \end{tikzcd}
    \quad
    \begin{tikzcd}
     \D^b_G(X_1) \ar[r,"\widetilde{\Phi}"] \ar[d,"\pi_{1}^*"] & \D^b_G(X_2) \ar[d,"\pi_{2}^{*}"] \\
    \D^b(X_1) \ar[r, "\Phi"] & \D^b(X_2),
    \end{tikzcd}
    \]
    where $\pi_i \colon X_i \to [X_i/G]$ are  structure morphisms of quotients. We may also call $\Phi$ {\it a lift} of $\widetilde{\Phi}$.
\end{defn}
\begin{rmk}
The push-forward and pull-back functors $\pi^*_i, \pi_{i,*}$ are both exact under the assumption that $(|G|,p)=1$ (\cf\cite[Lemma 3.8]{Elagin2014}). Thus the functors in Definition  \ref{defn:liftdescent} are all exact. 
\end{rmk}
Consider the $k$-linear triangluated category $\cT = \D^b_G(X)$. The non-trivial line bundles $\cL_{\chi}$ on $[X/G]$ induce a natural $\widehat{G}$-action on $\cT$ by taking tensor product (\cf Definition \ref{def:dualaction}). This leads to the following definition.
\begin{defn}\label{defn:equivariantfunctor}
Let $\cT_1$ and $\cT_2$ be two $k$-linear triangulated categories with $G_1$ and $G_2$ actions respectively.  Let $\gamma \colon G_1 \xrightarrow{\cong} G_2$ be a group isomorphism. An exact functor $\Phi\colon \cT_1 \to \cT_2$ is called $\gamma$-twisted equivariant if 
    \[
    \Phi \circ g^* \simeq \gamma(g)^* \circ \Phi \quad \text{for all } g \in G_1 .
    \]
    When $\gamma = \id_G$, the $\Phi$ will be called $G$-equivariant for simplicity.
If a descent $\widetilde{\Phi}$ of $\Phi$ is $\widehat{G}$-equivariant as an exact functor, then it will be called a \emph{$\widehat{G}$-equivariant descent} of $\Phi$.
\end{defn}

The following examples of equivariant functors is the most frequently used throughout this paper.
Let $\gamma\colon G \xrightarrow{\cong} G$ be an abstract group isomorphism, then there is an action of $G$ on $X_1 \times_k X_2$ (or $\D^b(X_1 \times X_2)$) by
\[
g \cdot (x_1,x_2) = (g\cdot x_1,\gamma(g) \cdot x_2).
\]
For instance, if $\gamma= \id$, then this action is just the diagonal action of $G$. The $G$-equivariant derived category of $X_1 \times_k X_2$ under the action given by $\gamma$ is denoted by $\D^b_{G_{\gamma}}(X_1 \times X_2)$. If there is an object $\cP=(P,\lambda)$ in $\D^b_{G_\gamma}(X_1 \times X_2)$ such that 
\[\Phi \simeq \mathbf{R}\!p_{2,*}(p_1^*(-)\otimes \cP),
\] then $\Phi$ is a $\gamma$-twisted $G$-equivariant exact functor. A $\gamma$-twisted $G$-equivariant functor of this form is called {\it integral} and will be denoted by $\Phi^{\cP}$, where the object $\cP$ is called its \emph{kernel}. In this case, we can forget the linearization  of the kernel $\cP$, which will defines an integral functor
\[
\Phi^{P} \colon \D^b(X_1) \to \D^b(X_2).
\]
It is not hard to check that $\Phi^{P}$ is a lift of $\Phi^{\cP}$.

A well-known fact is that a $G$-equivariant descent of a Fourier-Mukai transform is still an exact equivalence (\cf\cite[Proposition 3.10]{KrugSosna2015} or \cite[Lemma 5]{ploog07}). 
\begin{prop}\label{prop:equivalencelift}
Let $\cP=(P,\lambda)$ be a $G_{\gamma}$-equivariant object in $\D^b_{G_{\gamma}}(X_1 \times X_2)$ for some abstract isomorphism $\gamma\colon G \xrightarrow{\cong} G$. If $\Phi^{P}_{X_1 \to X_2}$ is an exact equivalence, then the integral functor  
\[\Phi\coloneqq\Phi^{\cP}\colon \D^b_G(X_1) \to \D^b_G(X_2)\] 
will also be an exact equivalence.
\end{prop}
The proof in \loccit also works for any algebraically closed field $k$ such that $(|G|,p)=1$. We sketch it here for  reader's convenience.
\begin{proof}

Take the kernel of right adjoint inverse of $\Phi$: $$Q= P^{\vee} \otimes p_{X_1}^* \omega_{X_1}[\dim X_1].$$ It admits a natural $G$-lineaization $\lambda'$ induced by $\lambda$. Denote $(Q,\lambda')$ by $\cQ$.
Note that the pull-back functor $\pi^*$ is just the forgetful functor of linearization.  Therefore $\pi^*(\cQ \circ \cP) \cong Q \circ P\cong  \cO_{\Delta_{X_1}}$. On the other hand, we have $\Aut(\cO_{\Delta_{X_1}}) = k^*$ in $\D^b(X_1 \times X_1)$.
Via computing the cohomology of $G$, we can show that the set of linearizations of $\cO_{\Delta_{X_1}}$ forms a principal homogeneous space under the action of $\widehat{G}=\Hom(G,k^*)$ (\cf\cite[Lemma 1]{ploog07}). In particular, since $\cO_{\Delta_{X_1}}$ admits the trivial linearization, the set of its linearizations is equal to $\widehat{G}$. It means that there is some $\chi \in \widehat{G}$ such that $\cQ \circ \cP \cong \Delta_*\cL_{\chi}$. Since tensoring line bundles will induce a derived equivalence, we can conclude that $\Phi$ is an exact equivalence.
\end{proof}
Suppose $G$ is a cyclic group. Then $\rH^2(G,k^*) =\rH^2(\widehat{G},k^*)=0$ by a direct computation. This implies that any Fourier-Mukai kernel $P$ (resp.~$\cP=(P,\lambda)$) admits a $G$-linearization (resp.~$\widehat{G}$-linearization). Now we can lift a $G$-equivariant derived equivalence along cyclic Galois coverings, i.e.~morphisms of smooth $k$-stacks $Y_i \to X_i=Y_i/G$\footnote{The quotient here is in stack-theoretical sense (\cf\cite{romagny05}).}.

\begin{cor}[{\cite[Proposition 4.3]{KrugSosna2015}}]\label{cor:cyclicdescent}
Let $Y_1 \to X_1$ and $Y_2 \to X_2$ be two cyclic Galois coverings. Suppose that there is an isomorphism $\gamma \colon \widehat{G} \xrightarrow{\cong} \widehat{G}$ . Then any $\gamma$-twisted equivariant Fourier-Mukai transform 
\[
\Phi^{\cP} \colon \D^b(X_1) \xrightarrow{\sim} \D^b(X_2)
\]
lifts to a derived equivalence
\[
\widetilde{\Phi} \colon \D^b(Y_1) \xrightarrow{\sim} \D^b(Y_2).
\] \qed
\end{cor}
\subsection{Proof of Theorem 1.2 (i) and (ii)}

\begin{prop}[Proposition 3.1, \cite{Stellari2007}]\label{prop:stellari1} 
If there is a Fourier-Mukai transform
\[
\Phi^{P} \colon \D^b(A_1) \xrightarrow{\sim} \D^b(A_2),
\]
then $\D^b([A_1/\iota]) \simeq\D^b([A_2/\iota])$.
\end{prop}
\begin{proof}
We repeat Stellari's proof for this statement here.
Let $\gamma$ be the abstract isomorphism $ [-1]_{A_1}  \mapsto[-1]_{A_2}$.
Let $\tau$ be the generator of $G_{\gamma}$. Let $T_{(x,y)}$ be the translation of a point $(x,y)$ on $A_1 \times A_2$. As $\Phi^{\tau^* P} \in \text{Eq}(\D^b(A_1),\D^b(A_2))$, by \cite[Corollary 3.4]{orlov02}, we have 
\[
\tau^* P \cong T_{(a,0),*} P \otimes p_{A_1}^* \cL_{\alpha}
\]
for some $a \in A_1(k)$, line bundle $\cL_{\alpha} \in \Pic^0(A_1) \cong \widehat{A}_1$. 

It is well-known that the Fourier-Mukai kernel $P$ is isomorphic to $\cE[i]$ for some semi-homogeneous sheaf $\cE$ on $A_1 \times A_2$, which implies 
\[
T_{(a,0),*}P \cong T_{(a,0),*} \cE [i] \cong P \otimes \cL_{\beta}
\]
for some $\cL_{\beta} \in \Pic^0(A_1 \times A_2)(k)$. Thus $\tau^* P \cong P \otimes \cL$, where $\cL= \cL_{\beta} \otimes p_A^*\cL_{\alpha}[i]$ . As $\Pic^0(A_1\times A_2)(k)$ is divisible, we can find a line bundle $\cN$ on $A_1 \times A_2$ such that $\cN^2 = \cL$. Let $\widetilde{P} = P \otimes \cN$. We can see
\begin{equation}\label{eq:tildeP}
\tau^* \widetilde{P} \cong \tau^* P \otimes \tau^* N \cong P \otimes \cL \otimes N^{\vee} \cong \widetilde{P}.
\end{equation}
On the other hand, since $\widetilde{P}$ is obtained from $P$ by tensoring line bundles, it also induces a Fourier-Mukai transform from $A_1$ to $A_2$. The isomorphism \eqref{eq:tildeP} ensures that $\widetilde{P}$ is $G_{\Delta}$-invariant, which also implies $\widetilde{P}$ admits a $G_{\gamma}$ linearization $\lambda$ as $\rH^2(G_{\gamma}, k^*)=0$. Therefore $\D^b([A/\iota]) \simeq \D^b([B/\iota])$ by Proposition \ref{prop:equivalencelift}.
\end{proof}

Combining the identification in Corollary \ref{cor:equivariantabelian}, the statement (i) in Theorem \ref{mainthm1} can be concluded. 

Krug and Sosna proved the converse direction for odd dimensional abelian varieties (\cf\cite[Proposition 5.13]{KrugSosna2015}). Their statement there is over $\mathbb{C}$, while we extend it to algebraically closed fields in positive characteristic.
\begin{prop}
Let $A_1$ and $A_2$ are two abelian varieties in dimension $2n+1$, $n \in \mathbb{Z}_{>0}$. Any Fourier-Mukai transform between Kummer stacks
\[
\Phi^{\cP} \colon \D^b([A_1/\iota]) \xrightarrow{\sim} \D^b([A_2/\iota])
\]
has a lift
\[
\widetilde{\Phi} \colon \D^b(A_1) \xrightarrow{\sim} \D^b(A_2).
\]
\end{prop}
\begin{proof}
The stack $[A/\iota]$ has non-trivial torsion canonical bundle when $\dim A$ is odd (\cf\cite{KrugSosna2015}). 
Let $X_i = [A_i /\iota]$.
We can take the abstract isomorphism $\langle K_{X_1}\rangle \xrightarrow[\sim]{\gamma} \langle K_{X_2} \rangle$, whose actions are given by tensoring. Notice that $K_{X_i} \otimes (-)$ is just the shift of Serre functor $S_i[-\dim_{X_i}]$ on $\D^b(X_i)$. Therefore $\Phi^{\cP}$ is $\gamma$-twisted equivariant as Serre functor commutes with all Fourier-Mukai transforms. Thus any $\D^b([A_1/\iota]) \xrightarrow{\sim} \D^b([A_2/\iota])$ can be lifted to an exact equivalence $\D^b(A_1) \simeq \D^b(A_2)$ by the Corollary \ref{cor:cyclicdescent}.
\end{proof}

This proves  Theorem \ref{mainthm1} (ii).

\begin{rmk} 
We wonder if the Kummer stacks $[A_1/\iota] \cong [A_2/\iota]$ will imply $\D^b(A_1) \cong \D^b(A_2)$ or not. The same question can be asked for two pairs $(X_1,G)$ and $(X_2,G)$ satisfying \texttt{(BKR1)} and \texttt{(BKR2)} such that  their quotients have isomorphic crepant resolution.
\end{rmk}

\subsection{Lifting Kummer structures}\label{subsec:liftingKummerstructure}

 To prove the Theorem (iii), we need to use the lifting of Kummer surfaces. A technical lemma is 
\begin{lemma}\label{lemma:liftingkummer}
Let $X$ be a K3 surface over $k$. Suppose $\cX$ is a relative K3 surface, lifting $X$ to some finite extension $V$ of the ring of Witt vectors $W(k)$, such that the specialization map is an isomorphism
\begin{equation}\label{lift-ns}
    \NS(\cX_{F}) \xrightarrow{\sim} \NS(X).
\end{equation}
For any abelian surface $A$ satisfying $\km(A) \cong X$, there is an abelian scheme $\cA$ over $V$ such that $\cA_{k} \cong A$ and $\cX \cong \km(\cA)$.
\end{lemma}
\begin{proof}
We may follow the idea of \cite[Proposition 1.1]{shioda78} to prove the existence of the lifting.   Let $E_a$ be the exceptional curve on $X$ with respect to the $2$-torsion point $a \in A[2]$. It has a unique Cartier divisor extension $\cE_a \in \Pic(\cX)$ by the construction. Moreover, the divisor
    \[
    \cE \coloneqq \sum_{a \in A[2]} \cE_{a} \in \Pic(\cX)
    \]
    is $2$-divisible as $\cE$ is equal to the extension of $\frac{1}{2} \sum_{a \in A[2]} E_a \in \NS(X)$ via \eqref{lift-ns}. Therefore, we have a double covering of $V$-schemes
    \[
    \pi \colon \cY \longrightarrow \cX,
    \]
    which is branched along $\cE \subset \cX$. We can contract the curves $\cC_a \subset \cY$ lying over $\cE_a$ to points:
    \[
    f \colon \cY \longrightarrow \cA,
    \]
    where $\cA$ is a smooth proper scheme over $V$. Thus we can see $\cA$ is the required abelian scheme since $\cX \cong \cY/ \iota \cong \km(\cA)$ where $\iota$ is the involution of $\cY$ induced by the double covering $\pi$.
\end{proof}

\subsection{Derived Torellli theorem and specialization}
Let $A_1$ and $A_2$ be two abelian varieties. Consider the set  of \emph{symplectic isomorphisms}
\[
U(A_1,A_2)=\left\{ f = \begin{psmallmatrix}
f_1 & f_2\\
f_3& f_4
\end{psmallmatrix}\colon A_1 \times \widehat{A}_1 \xrightarrow{\sim}  A_2 \times \widehat{A}_2 \Big| f^{-1}= \widetilde{f} \coloneqq \begin{psmallmatrix}
\hat{f_4} & -\hat{f_2} \\
-\hat{f_3} & \hat{f_1}
\end{psmallmatrix}\right\}.
\]
Let $\text{Eq}(\D^b(A_1), \D^b(A_2))$ be the set of exact equivalences from $\D^b(A_1)$ to $\D^b(A_2)$. Then the derived Torelli theorem for abelian varieties asserts 
\begin{thm}[Orlov--Polishchuk]\label{thm:OrlovPolischuk}
There is a morphism
\begin{equation}
\label{eq:symplectic}
\gamma_{A_1,A_2}\colon \text{Eq}\left(\D^b(A_1),\D^b(A_2)\right) \to U(A_1,A_2).
\end{equation}
If $U(A_1,A_2) \neq \emptyset$, then $A_1$ and $A_2$ are Fourier-Mukai partners (\cf\cite{orlov02} or \cite{Polishchuk2003}).
\end{thm}

The following lemma shows that derived equivalence is preserved under smooth specialization. 
\begin{lemma}\label{lemma:specializationofDE}
For two abelian schemes $\cA_1$ and $\cA_2$ over $V$, any derived equivalence $\D^b(\cA_{1,F}) \xrightarrow{\sim} \D^b(\cA_{2,F})$ has a specialization $\D^b(\cA_{1,k}) \to \D^b(\cA_{2,k})$, which is also a derived equivalence.
\end{lemma}
\begin{proof}
It is known that any derived equivalence $\D^b(\cA_{1,F}) \xrightarrow{\sim} \D^b(\cA_{2,F})$ induces a symplectic isomorphism $f_F \colon \cA_{1,F} \times \widehat{\cA}_{1,F} \xrightarrow{\sim}\cA_{2,F} \times \widehat{\cA}_{2,F}$.
By using the Matsusaka-Mumford theorem (\cf\cite[Chapter I. Corollary 1]{MatsusakaMumford}), $f_F$ can be extended to an isomorphism $f \colon \cA_1 \times \widehat{\cA}_1 \xrightarrow{\sim}\cA_2 \times \widehat{\cA_2}$ such that the restriction $f_k$ is a specialization of $f_F$. 
On the other hand, we can also take the specialization of $\widetilde{f}_F$. As $(\widetilde{f} \circ f)_F = \widetilde{f}_F \circ f_F = \id_{\cA_{1,F} \times \widehat{\cA}_{1,F}}$. The graph of $\widetilde{f} \circ f$ is the diagonal of $\cA \times \widehat{\cA}$. Thus the specialization $f_k$ is also a symplectic isomorphism. Hence the derived equivalence $\D^b(\cA_{1,k}) \xrightarrow{\sim} \D^b(\cA_{2,k})$ is from Theorem \ref{thm:OrlovPolischuk}.
\end{proof}
  
\subsection{Proof of Theorem \ref{mainthm1} (iii): finite height case} 
 Given two abelian surfaces $A_1$ and $A_2$, if there is an exact derived equivalence $\D^b(A_1) \simeq \D^b(A_2)$, then we have $\D^b(\km(A_1)) \cong \D^b(\km(A_2))$ by Proposition \ref{prop:stellari1}. Then we have 
$$\km(A_1)\cong \km(A_2)$$
because the Kummer surface does not have non-trivial Fourier-Mukai partners ({\it cf.}~\cite[Theorem 1.1]{LO15}). 

Recall that an abelian surface is called of finite height if it is not supersingular. Assume that $A_1$ and $A_2$ are of finite heights. In this case,  if $X$ is a Kummer surface isomorphic to $\km(A_1)\cong \km(A_2)$, then $X$ is a K3 surface with finite height.  There is a N\'eron-Severi lattice preserving lifting $\cX$ of $X$ over some $V$ (\cf\cite[Corollary 4.2]{LieblichMaulik}). By Lemma \ref{lemma:liftingkummer}, there exist relative abelian surfaces $\cA_1$ and $\cA_2$ over $V$ such that $$\km(\cA_1)\cong \cX\cong \km(\cA_2).$$ 
Due to \cite[Theorem 0.1 (1)]{HosonoLianEtAl2003}, the generic fibers of $\cA_1$ and $\cA_2$ are geometrically derived equivalent. By using the standard spreading out argument and \cite[Lemma 2.12]{orlov02}, we can find a finite extension $V'$ of $V$ such that $\D^b(A_{1,F'}) \xrightarrow{\sim} \D^b(A_{2,F'})$, where $F'$ is the fraction field of $V'$. The assertion  follows from Lemma \ref{lemma:specializationofDE}.

\section{Supersingular derived Torelli theorem}\label{sec:FMpartnersabelian}
In this section, we will focus on derived categories of supersingular abelian varieties. There is a cohomological realization of the derived Torelli theorem for  supersingular abelian varieties using  supersingular abelian crystals and supersingular K3 crystals. 

Throughout the rest part, we fix an algebraically closed field $k$ in characteristic $p$. We also denote $W=W(k)$ for the ring of Witt vectors of $k$ and $\sigma \colon W \to W$ for the Frobenius morphism.

\subsection{Abelian crystals and K3 crystals} For supersingular abelian varieties, we can give a cohomological realization of Theorem \ref{thm:OrlovPolischuk}  via Ogus' supersingular Torelli theorem (\cite{Og79}).

Let $A$ be an abelian variety of dimension $g$ over $k$.  Recall that $A$ is {\it supersingular} if the crystalline cohomology $\rH_{\crys}^1(A/W)$ is purely of slope $\frac{1}{2}$ as a weight one $F$-crystal with natural Frobenius. Over an algebraically closed field, it is also equivalent to say that $A$ is isogenous to a $g$-fold product of supersingular elliptic curves (\cf\cite[Theorem (4.2)]{oort74}). 
\begin{defn}
An abelian crystal of genus $g$ is a $F$-crystal $(H,\varphi)$ of rank $2g$ endowed with an isomorphism between $F$-crystals $\tr\colon \Lambda^{2g}H \to W(-n)$, whose Hodge numbers of Hodge polygon are both equal to $g$. The isomorphism $\tr$ is called the trace of abelian crystal $(H,\varphi)$.
\end{defn}

As its name indicates, for a $g$-dimensional abelian variety $A$, the $F$-crystal $\rH^1_{\crys}(A/W)$ is an example of abelian crystal. The trace map of $\rH^1_{\crys}(A/W)$ is given by 
\[
\Lambda^{2g} \rH^1_{\crys}(A/W) \cong \rH^{2g}_{\crys}(A/W) \cong W(-n).
\]
Here the first isomorphism is coming from the multiplication of $A$ and the cup-product of its crystalline cohomology. 

The crystalline Torelli theorem of supersingular abelian varieties states that, for any integer $g \geq 2$, there is a bijection
\begin{equation}
    \left\lbrace \begin{matrix}\text{isomorphism classes of } \\
    \text{ supersingular abelian varieties}\\
    \text{of genus $g$}\end{matrix} \right\rbrace \xrightarrow[H^1_{\crys}(-/W)]{\sim} \left\lbrace \begin{matrix}
    \text{isomorphism classes of} \\
    \text{supersingular abelian crystals} \\
    \text{ of genus $g$}\end{matrix} \right\rbrace_.
\end{equation}
See \cite[Theorem 6.2]{Og79}. 

The second crystalline cohomology of an abelian surface is also equipped with a structure called \emph{K3-crystal} which is defined as follows.
\begin{defn}
	We say a $F$-crystal $(H,\varphi)$ is K3-crystal with rank $n$, if there is a symmetric bilinear form $\langle, \rangle_H$ on $H$, satisfying:  
	\begin{enumerate}
	    \item  $H$ is of weight $2$, that means $p^2H \subset \im(\varphi)$; 
	    \item the Hodge number $h^0(H)=1$, that means $\varphi \otimes \id_k$ is of rank 1; 
	    \item $\langle, \rangle_H$ is perfect; 
	    \item $\langle \varphi x, \varphi y \rangle_H = p^2 \sigma\langle x, y \rangle_H$ for any $x,y \in H$.
	\end{enumerate}
Inside a K3-crystal $(H,\varphi)$, there is a natural $\mathbb{Z}_p$-lattice defined as
\[
T_H \coloneqq H^{\varphi =p} =\left\{ x \in H  \big| \varphi(x) = px \right\}
\]
equipped with bilinear form $\langle,\rangle_{T_H}$ induced from $H$. 
A K3-crystal is called \emph{supersinguar} if it is purely of slope 1. If $(H,\varphi)$ is supersingular of rank $n$, then its Tate module $T_H$ is a free $\Z_p$-module of rank $n$, which comes from the definition of pure of slope $1$. Its discriminant is equal to $p^{-2\sigma_0}$ for some integer $\sigma_0 \geq 1$, which is called the \emph{Artin invariant} of $H$.
\end{defn}

Let $A$ be an abelian surface over $k$. Then $\rH^2_{\crys}(A/W)$ is naturally endowed with a K3-crystal structure. Consider the canonical isomorphism
\[\rH^2_{\crys}(A/W) \cong \Lambda^2\rH^1_{\crys}(A/W),
\]
one can see that $A$ is supersingular if and only if $\rH^2_{\crys}(A/W)$ is supersingular as a K3-crystal. An important observation due to Ogus is that the functor $\Lambda^2$ forms an equivalence from the category of supersingular abelian crystals of genus 2 to the category of supersingular K3 crystals of rank 6 (cf.\ \cite[Proposition 6.9]{Og79}). Thus the crystalline Torelli theorem for supersingular abelian surfaces can be rephrased in terms of K3-crystals, that means
two supersingular abelian surfaces $A_1$ and $A_2$ are isomorphic if and only if
\[
\rH^2_{\crys}(A_1/W)\cong \rH^2_{\crys}(A_2/W).
\]
\begin{rmk}\label{rmk:pp}
This also implies that any supersingular abelian surface $A$ is principal polarized, since there is a natural isomorphism of K3-crystals $\rH^2_{\crys}(A/W) \xrightarrow{\sim} \rH^2_{\crys}(\widehat{A}/W)$. This can also deduced from the classification of N\'eron-Severi lattices of supersingular abelian surfaces (\cf\cite[Lemma 6.2]{FL18}).
\end{rmk}

One can characterize $K3$-crystals via the characteristic space.  Let $H$ be a supersingular K3 crystal with Artin invariant $\sigma_0$. We have the an orthogonal decomposition (not unique!) for its Tate module:
\begin{equation}\label{eq:decompositonTatemodule}
(T_H, \langle,\rangle_{T_H}) \cong (T_0, p \langle, \rangle_{T_0}) \oplus (T_1, \langle, \rangle_{T_1}),
\end{equation}
such that $T_0$ is of rank $2\sigma_0$, $\langle, \rangle_{T_0}$ and $\langle, \rangle_{T_1}$ are both perfect, since the cokernel of $T_H \hookrightarrow T_H^{\vee}$ is killed by $p$. The kernel 
\[
\overline{H} \coloneqq \ker (T_H \otimes_{\mathbb{Z}_p} k \to H\otimes_W k)
\]
forms a $\sigma_0$-dimensional $\mathbb{F}_p$-vector space which is totally isotropic with respect ot $\langle,\rangle_{T_H}$ and it is isomorphic to the image of
\[
H \cong H^{\vee} \to T_H^{\vee}\otimes_{\mathbb{Z}_p} W
\]
It also forms a \emph{strictly characteristic subspace} of $T_0 \otimes_{\mathbb{Z}_p} k$ (\cf\cite[Definition 3.19]{Og79} for the definition and {\loccit~Remark 3.16} for the explanation). Let $\mathbb{K}_H = \varphi^{-1}(\overline{H}) \subset T_0 \otimes_{\mathbb{Z}_p} k$, which is another strictly characteristic subspace of $T_0\otimes_{\mathbb{Z}_p} k$. We also call it the characteristic space of the $K3$-crystal $H$. 

The classification of $K3$-crystals (see \cite[Theorem 3.20]{Og79}) asserts that
\begin{thm}[Ogus]\label{thm:characteristicspace}
For two supersingular $K3$-crystals $H$ and $H'$ of the same rank, $H \cong H'$ if and only if there is an isomorphism of pairs $(T_0 \otimes k, \mathbb{K}_H) \cong (T_0' \otimes k, \mathbb{K}_{H'})$.
\end{thm}

\subsection{Supersingular derived Torelli theorem} \label{subsec:supersingularderivedtorelli}
Let $\cP$ be the Poincar\'e line bundle on $A \times \widehat{A}$ associated to the polarization $\varphi_{\cL}$. There is a canonical isomorphism between Dieudonn\'e modules:
\begin{equation}\label{eq:dieudual}
\Phi_A \colon \rH^1_{\crys}(A/W)^* \xrightarrow{} \rH^1_{\crys}(\widehat{A}/W),
\end{equation}
which corresponds to the first crystalline Chern class
\[
c_1(\cP) \in \rH^1_{\crys}(A/W) \otimes_W \rH^1_{\crys}(\widehat{A}/W) \subseteq \rH^2_{\crys}(A \times \widehat{A}/W);
\]
see \cite[Th\'eor\`em 5.1.2]{BBM2006}. We have a natural quadratic form on the $F$-crystal $\rH^1_{\crys}(A \times \widehat{A})$:
\[
q_A(a, \alpha) = 2 \Phi_A^{-1}(\alpha)(a) \quad \text{for } (a,\alpha) \in \rH^1_{\crys}(A/W) \oplus \rH^1_{\crys}(\widehat{A}/W)
\]
With abelian crystals, we are able to provide a cohomological description of Theorem \ref{thm:OrlovPolischuk} for supersingular abelian varieties.

It is clear that any symplectic isomorphism $f \colon A_1\times \widehat{A}_1 \xrightarrow{\sim} A_2 \times \widehat{A}_2$ will induce an isomorphism of abelian crystals
\[
\varphi \colon \rH^1_{\crys}(A_1/W) \oplus \rH^1_{\crys}(\widehat{A}_1/W) \xrightarrow{\sim}\rH^1_{\crys}(A_2/W) \oplus  \rH^1_{\crys}(\widehat{A}_2/W),
\]
by taking K\"unneth decomposition.

Conversely, we can write the isomorphism $\varphi$ into a $2 \times 2$-matrix 
\[
\rH^1_{\crys}(A_1/W) \oplus \rH^1_{\crys}(\widehat{A}_1/W) \xrightarrow{
\begin{pmatrix}
\varphi_1 & \varphi_2 \\
\varphi_3 & \varphi_4
\end{pmatrix}}\rH^1_{\crys}(A_2/W) \oplus  \rH^1_{\crys}(\widehat{A}_2/W)
\]
in which $\varphi_i$ are all morphisms between $F$-crystals. Then $\varphi$ being isometry in terms of the quadratic form $q_A$ is equivalent to satisfying matrix equation
\[
\begin{pmatrix}
\varphi_1^t & \varphi_3^t \\
\varphi_2^t & \varphi_4^t
\end{pmatrix}
\begin{pmatrix}
0 & 1 \\
1 & 0
\end{pmatrix}
\begin{pmatrix}
\varphi_1 & \varphi_2 \\
\varphi_3 & \varphi_4
\end{pmatrix}
= \begin{pmatrix}
0 & 1 \\
1 & 0
\end{pmatrix}.
\]
It implies that
\[
\varphi^{-1}= \begin{pmatrix}
0 & 1 \\
1 & 0
\end{pmatrix}
\begin{pmatrix}
\varphi_1^t & \varphi_3^t \\
\varphi_2^t & \varphi_4^t
\end{pmatrix}
\begin{pmatrix}
0 & 1 \\
1 & 0
\end{pmatrix}=
\begin{pmatrix}
\varphi_4^t & \varphi_2^t \\
\varphi_3^t & \varphi_1^t
\end{pmatrix}
\]
Now suppose $A_1$ or $A_2$ is supersingular. By Ogus's crystalline Torelli theorem for supersinsgular abelian varieties (cf.\ \cite[Theorem 6.2]{Og79}), we can find an isomorphism $f \in \Hom(A_1\times \widehat{A}_1, A_2\times \widehat{A}_2)$ such that $\rH^1_{\crys}(f) = \varphi$ . It also satisfies $H^1_{\crys}(f_i)= \varphi_i$ if we rewrite $f$ into the following matrix form uniquely:
\[
f=\begin{pmatrix}
f_1 & f_2 \\
f_3 & f_4
\end{pmatrix}.
\]
Note that
\[
c_1(\widehat{\cP}) = - c_1(\cP) \in \rH^1_{\crys}(\widehat{A}/W) \otimes \rH^1_{\crys}(A/W)
\]
by the identification $\rH^1_{\crys}(\widehat{\widehat{A}}/W) \cong \rH^1_{\crys}(A/W)$. Thus we have
\[
\varphi_1^t = \rH^1_{\crys}(\widehat{f}_1) \quad \varphi_2^t = - \rH^1_{\crys}(\widehat{f}_2) \quad \varphi_3^t = - \rH^1_{\crys}(\widehat{f}_3) \quad \varphi_4^t = \rH^1_{\crys}(\widehat{f}_4).
\]
Thus $f \in U(A_1,A_2)$. Therefore, we have the following
\begin{prop}
For arbitrary supersingular abelian varieties $A_1$ and $A_2$ over $k$, the following statements are equivalent.
\begin{enumerate}
    \item $\D^b(A_1) \simeq \D^b(A_2)$.
    \item There is a isometry of abelian crystals
    \[
    \rH^1_{\crys}(A_1 \times \widehat{A}_1/W) \xrightarrow{\sim} \rH^1_{\crys}(A_2 \times \widehat{A}_2/W).
    \]
    with respect to $q_{A_i}$.
\end{enumerate}
\end{prop}

For supersingular abelian surfaces, we have the following consequence via K3-crystals.
\begin{thm}\label{thm:supersingularDerivedTorelli}
Let $A_1$ and $A_2$ be two supersingular abelian surfaces over $k$. Then
$\D^b(A_1)\cong \D^b(A_2)$  if and only if 
$A_1 \cong A_2$.
\end{thm}

\begin{proof}

It is clear that when $\dim A_1 = 2$ then $\dim A_2 =2$ if they are derived equivalent. The given derived equivalence $\Phi^P$ induces an isometry between K3 crystals $\widetilde{\rH}(A_1/W) \cong \widetilde{\rH}(A_2/W)$, where
\[
\widetilde{\rH}(A_i/W) \coloneqq W(-1) \oplus \rH^2_{\crys}(A_i/W) \oplus W(-1)
\]
is the Mukai K3-crystal of $A_i$ equipped with the Mukai pairing.
Since $\rH^2_{\crys}(A_i/W)^{\varphi_i =p} \cong \NS(A_i) \otimes \ZZ_p$ (\cf \cite[(1.6)]{Og79}), we have
\[
\widetilde{\rH}(A_i/W)^{\varphi_i =p} \cong \mathbb{Z}_p^{\oplus 2} \oplus \NS(A_i) \otimes \ZZ_p.
\]
Therefore, the characteristic subspace of $\widetilde{\rH}(A_i/W)$ is equal to
\[
\ker \left( \NS(A_i) \otimes k \to \rH^2_{\dR}(A_i/k) \right),
\]
which is isomorphic to the characteristic subspace of $\rH^2_{\crys}(A_i/W)$. Let 
\[
\NS(A_i) \otimes \mathbb{Z}_p \cong \left(T_0^{(i)}, p\langle, \rangle \right)\oplus \left(T_1^{(i)}, \langle, \rangle\right)
\]
be decomposition of $\mathbb{Z}_p$-lattice as in \eqref{eq:decompositonTatemodule}. Then there is a decomposition of the Tate module of $\widetilde{\rH}(A_i/W)$:
\[
\left(T_0^{(i)},p \langle, \rangle\right) \oplus \left(T_1^{(i)} \oplus \mathbb{Z}_p^{\oplus 2},\langle,\rangle \right),
\]
where the second inner product is from the restriction of Mukai pairing.
Therefore 
\[
(T_0^{(1)}, \mathbb{K}_{A_1}) \cong (T_0^{(2)}, \mathbb{K}_{A_2}),
\]
which implies an isomorphism between K3-crystals:
\[
\rH^2_{\crys}(A_1/W) \cong \rH^2_{\crys}(A_2/W).
\]
Then the isomorphism $A_1 \cong A_2$ follows from the crystalline Torelli theorem.
\end{proof}

Now we can give a summary of the known equivalence relations between supersingular abelian surfaces.
\begin{cor}\label{cor3}
Let $A_1, A_2$ be two abelian surfaces. If $A_1$ is supersingular,  then the following statements are equivalent:
\begin{TFAE}[leftmargin=1.5pc]
    \item There is an isomorphism $A_1 \cong A_2$;
    \item There is an isomorphism between K3-crystals $\rH^2_{\crys}(A_1/W) \cong \rH^2_{\crys}(A_2/W)$;
    \item There is an isomorphism between Kummer surfaces $\km(A_1) \cong \km(A_2)$;
    \item There is a derived equivalence $\D^b(\km(A_1)) \xrightarrow{\sim} \D^b(\km(A_2))$;
    \item There is a derived equivalence $\D^b(A_1) \xrightarrow{\sim} \D^b(A_2)$.
\end{TFAE}
\end{cor}
\begin{proof}
We firstly note that the supersingularity is invariant under the relations listed in the statements. Hence $A_2$ is supersingular in each statement.

Then the equivalence between (a), (b) is just the Ogus's crysalline Torelli theorem for supersingular abelian surfaces. The statements (c), (d) and (e) are equivalent by the Theorem \ref{mainthm1}. The (a) and (e) are equivalent by Theorem \ref{thm:supersingularDerivedTorelli}.
\end{proof}

Therefore, the statement (iii) in Theorem \ref{mainthm1} is true for supersingular abelian surfaces.
\begin{rmk}
We suspect whether every supersingular abelian variety  $A$ has only trivial (in the strong sense) Fourier-Mukai partners, i.e. $\mathrm{FM}(A)=\{A\}$. This requires that  $A$ admits a principal polarization, which holds if $\dim A=2$ (\cf Remark \ref{rmk:pp}). But it may fail in  higher dimensional case. See \cite[ \S 10]{LOo98} for the discussion of non-principal polarizations on supersingular abelian varieties of any genus. It will be very interesting to know if  $\mathrm{FM}(A)=\{A, \widehat{A} \}$. 
\end{rmk}

\subsection{Proof of Theorem \ref{thm:supersingulargeneralizedKummer}}
 
We can see there is a quasi-liftably birational map 
\begin{equation}\label{eq:birationalKumtype}
K_{v}(A)\dashrightarrow K_{n}(A'),
\end{equation}
with $n=\frac{v^{2}}{2}-1$, $p \nmid n$ and some (supersingular) abelian surface $A'$  derived equivalent to $A$ (\cf\cite[Theorem 6.12]{FL18}). Then (i) follows from the fact $A$ does not have non-trivial Fourier-Mukai partners (\cf\ref{cor3}). 

To prove (ii), our strategy is to endow $\rH^2_{\crys}(K_n(A)/W)$ with a natural\footnote{Here the ``natural" means the structure should be at least functorial with respect to algebraic correpsondences.} K3-crystal structure, which is closely related to that of $A$.  Combining the algebraic correspondence given by the birational map \eqref{eq:birationalKumtype}, one can show that there is an isomorphism between K3-crystals of $A$ and $A'$. In fact,
\begin{lemma}\label{lemma:BBforms}
Assume that $p \nmid n+1$.
\begin{enumerate}
    \item For any abelian surface $A$ and positive integer $n$, we can endow the $F$-crystal $\rH^2_{\crys}(K_n(A)/W)$ with the Beauville-Bogomolov form $q$, which makes $(\rH^2_{\crys}(K_n(A)/W),q)$ a K3 crystal of rank 7.
    \item There is an orthogonal decomposition with respect to $q$:
    \begin{equation}\label{eq:FcrystalDecom}
    \rH^2_{\crys}(K_n(A)/W) \cong \rH^2_{\crys}(A/W) \oplus W(-1).
    \end{equation}
\end{enumerate}

\end{lemma}
\begin{proof}
Let $\cA$ be a lifting of abelian surface $A$ to the Witt vector ring $W$, which is an abelian scheme over $W$ (\cf\cite[(11.1)]{oort85}). Consider the summation of points $s \colon \cA^{[n]}\to \cA $. The fiber of $s$ at the origin point $\spec(W) \to \cA$ is a lifting of generalized Kummer variety $K_n(A)$, which will be denoted by $\cK_n(\cA)$ as a $W$-scheme.
For the geometric generic fiber $\cK_n(\cA)_{\bar{K}}$, we have the Beauville-Bogomolov form $q_{\bar{K}}$ on $p$-adic \'etale cohomology $\rH^2_{\et}(\cK_n(\cA)_{\bar{K}},\ZZ_p)$. There is an orthogonal decomposition
\begin{equation}
    \rH^2_{\et}(\cK_n(\cA)_{\bar{K}}, \ZZ_p) \cong \rH^2_{\et}(\cA_{\bar{K}},\ZZ_p) \oplus \ZZ_p\delta,
\end{equation}
such that $q_{\bar{K}}(\delta) = -2(n+1)$.
The Beauville-Bogomolov form on $\rH^2_{\crys}(K_n(A)/W)$ can be defined by applying integral crystalline-\'etale comparison (\cf\cite[Theorem 1]{BMS1}) and denoted by $q$. It is followed by a decomposition of $F$-crystals
\[
\rH^2_{\crys}(K_n(A)/W) \cong \rH^2_{\crys}(A/W) \oplus W(-1),
\]
with respect to $q$.

The bilinear form induced by $q$ is perfect since $p \nmid n+1$. A direct computation shows that
\[
q(\varphi^*{\alpha})= p^2 \sigma q(\alpha).
\]
The canonical decomposition \eqref{eq:FcrystalDecom}
implies the Hodge number $h^0(\rH^2_{\crys}(K_n(A)/W)) =1$. Therefore, $(\rH^2_{\crys}(K_n(A)/W), q)$ is a K3-crystal.
\end{proof}

By the decomposition \eqref{eq:FcrystalDecom}, we have
\[
\rH^2_{\crys}(K_n(A)/W)^{\varphi =p} \cong \rH^2_{\crys}(A/W)^{\varphi =p} \oplus \mathbb{Z}_p \delta.
\]
Thus there is an isomorphism between Tate modules (as $\ZZ_p$-lattices) 
\begin{equation}\label{eq:decomposeNS}
\NS(K_n(A)) \otimes \mathbb{Z}_p \cong \left(\NS(A) \otimes \mathbb{Z}_p \right) \oplus \langle -2(n+1) \rangle.
\end{equation}

\begin{lemma}\label{lemma:abelian-Kummer}
Under the same assumption in \ref{lemma:BBforms}, if $A$ is supersingular, then we have an isomorphism of characteristic subspaces
\[
\mathbb{K}_{\rH^2_{\crys}(A/W)} \xrightarrow{\sim} \mathbb{K}_{\rH^2_{\crys}(K_n(A)/W)}.
\]
\end{lemma}
\begin{proof}
Consider the following commutative diagram
\[
\begin{tikzcd}
\NS(A) \otimes \mathbb{Z}_p \ar[r,"c_1"] \ar[d,hook] & \rH^2_{\crys}(A/W) \ar[d,hook] \\
\NS(K_n(A)) \otimes \mathbb{Z}_p\ar[r,"c_1"] & \rH^2_{\crys}(K_n(A)/W)
\end{tikzcd}
\]
The horizontal injective morphisms are isometric embeddings from \eqref{eq:decomposeNS}. Thus it is not hard to see that they induces isomorphic characteristic subspaces.
\end{proof}
\begin{rmk}
The Lemma \ref{lemma:abelian-Kummer} can be viewed as a higher dimensional analogue to the relationship between supersingular abelian surfaces and their associated supersingular Kummer surfaces, which is a miracle used in Ogus's proof of crystalline Torelli theorem for supersingular Kummer surfaces.
\end{rmk}

\begin{lemma}\label{lemma:isocrystalH2}
If $K_n(A)$ and $K_n(A')$ are quasi-liftably birational equivalent, then there is an isomorphism of $F$-crystals $\rH^2_{\crys}(K_n(A)/W) \xrightarrow{\sim} \rH^2_{\crys}(K_n(A')/W)$ preserving Beauville-Bogomolov forms.
\end{lemma}
\begin{proof}
We may assume that $K_n(A)$ and $K_n(A')$ are liftable as $\cK$ and $\cK'$ over some finite extension $V$ of $W$, whose geometric generic fibers are birational equivalent irreducible symplectic varieties. Then there is a correspondence $Z_{F} \subset (\cK_F \times \cK_F')$ for some totally ramified finite field extension $F$ of $K=\Frac(V)$, such that $[Z_F]_{*} \colon \rH^2_{\et}(\cK_{\bar{K}},\ZZ_p) \xrightarrow{\sim} \rH^2_{\et}(\cK'_{\bar{K}},\ZZ_p)$ is a $G_F$-equivariant isomorphism, compatible with Beauville-Bogomolov forms (\cf\cite[Lemma 2.6]{Huybrechts99}). The construction in Lemma \ref{lemma:BBforms} implies that there is an isomorphism of $F$-crystals as we required.
\end{proof}

If $K_v(A)$ and $K_{v'}(A')$ are quasi-liftably birational equivalent, then $K_n(A)$ and $K_n(A')$ are also quasi-liftably birational equivalent by combining \eqref{eq:birationalKumtype}.
This implies that there exists isomorphisms between characteristic subspaces
\[
\bK_{A} \xrightarrow[\text{Lemma \ref{lemma:abelian-Kummer}}]{\sim} \bK_{K_n(A)} \xrightarrow[\text{Lemma \ref{lemma:isocrystalH2}}]{\sim} \bK_{K_n(A')} \xrightarrow[\text{Lemma \ref{lemma:abelian-Kummer}}]{\sim} \bK_{A'}.
\]
Therefore, by Theorem \ref{thm:characteristicspace} there is an isomorphism of K3-crystals
\[
\rH^2_{\crys}(A/W) \cong \rH^2_{\crys}(A'/W).
\]
Now we can conclude that $A \cong A'$ by Ogus's crystalline Torelli theorem for supersingular abelian surfaces.

\bibliographystyle{amsplain}
\bibliography{refs}

\providecommand{\bysame}{\leavevmode\hbox to3em{\hrulefill}\thinspace}
\providecommand{\MR}{\relax\ifhmode\unskip\space\fi MR }
\providecommand{\MRhref}[2]{%
  \href{http://www.ams.org/mathscinet-getitem?mr=#1}{#2}
}
\providecommand{\href}[2]{#2}
\begin{thebibliography}{10}

\bibitem{BBM2006}
Pierre Berthelot, Lawrence Breen, and William Messing, \emph{Th\'{e}orie de
  {D}ieudonn\'{e} cristalline. {II}}, Lecture Notes in Mathematics, vol. 930,
  Springer-Verlag, Berlin, 1982. \MR{667344}

\bibitem{BMS1}
Bhargav Bhatt, Matthew Morrow, and Peter Scholze, \emph{Integral {$p$}-adic
  {H}odge theory}, Publ. Math. Inst. Hautes \'{E}tudes Sci. \textbf{128}
  (2018), 219--397. \MR{3905467}

\bibitem{BKR}
Tom Bridgeland, Alastair King, and Miles Reid, \emph{The {M}c{K}ay
  correspondence as an equivalence of derived categories}, J. Amer. Math. Soc.
  \textbf{14} (2001), no.~3, 535--554. \MR{1824990}

\bibitem{dolgachevnote}
Igor~V. Dolgachev, \emph{Derived categories},
  \url{www.math.lsa.umich.edu/~idolga/derived9.pdf}.

\bibitem{Elagin2014}
Alexey Elagin, \emph{On equivariant triangulated categories}, arXiv:1403.7027
  (2014).

\bibitem{FL18}
Lie Fu and Zhiyuan Li, \emph{Supersingular irreducible symplectic varieties},
  {R}ationality of {A}lgebraic {V}arieties, Progress in Mathematics, vol. 342,
  2021, pp.~191--244.

\bibitem{HosonoLianEtAl2003}
Shinobu Hosono, Bong~H. Lian, Keiji Oguiso, and Shing-Tung Yau, \emph{Kummer
  structures on {$K3$} surface: an old question of {T}. {S}hioda}, Duke Math.
  J. \textbf{120} (2003), no.~3, 635--647. \MR{2030099}

\bibitem{Huybrechts99}
Daniel Huybrechts, \emph{Compact hyper-{K}\"{a}hler manifolds: basic results},
  Invent. Math. \textbf{135} (1999), no.~1, 63--113. \MR{1664696}

\bibitem{KrugSosna2015}
Andreas Krug and Pawel Sosna, \emph{Equivalences of equivariant derived
  categories}, J. Lond. Math. Soc. (2) \textbf{92} (2015), no.~1, 19--40.
  \MR{3384503}

\bibitem{LOo98}
Ke-Zheng Li and Frans Oort, \emph{Moduli of supersingular abelian varieties},
  Lecture Notes in Mathematics, vol. 1680, Springer-Verlag, Berlin, 1998.
  \MR{1611305}

\bibitem{LieblichMaulik}
Max Lieblich and Davesh Maulik, \emph{A note on the cone conjecture for {K}3
  surfaces in positive characteristic}, Math. Res. Lett. \textbf{25} (2018),
  no.~6, 1879--1891. \MR{3934849}

\bibitem{LO15}
Max Lieblich and Martin Olsson, \emph{Fourier-{M}ukai partners of {K}3 surfaces
  in positive characteristic}, Ann. Sci. \'{E}c. Norm. Sup\'{e}r. (4)
  \textbf{48} (2015), no.~5, 1001--1033. \MR{3429474}

\bibitem{MatsusakaMumford}
T.~Matsusaka and D.~Mumford, \emph{Two fundamental theorems on deformations of
  polarized varieties}, Amer. J. Math. \textbf{86} (1964), 668--684.
  \MR{171778}

\bibitem{Og79}
Arthur Ogus, \emph{Supersingular {$K3$} crystals}, Journ\'{e}es de
  {G}\'{e}om\'{e}trie {A}lg\'{e}brique de {R}ennes ({R}ennes, 1978), {V}ol.
  {II}, Ast\'{e}risque, vol.~64, Soc. Math. France, Paris, 1979, pp.~3--86.
  \MR{563467}

\bibitem{Olsson2016}
Martin Olsson, \emph{Algebraic spaces and stacks}, American Mathematical
  Society Colloquium Publications, vol.~62, American Mathematical Society,
  Providence, RI, 2016. \MR{3495343}

\bibitem{oort74}
Frans Oort, \emph{Subvarieties of moduli spaces}, Invent. Math. \textbf{24}
  (1974), 95--119. \MR{424813}

\bibitem{oort85}
\bysame, \emph{Lifting algebraic curves, abelian varieties, and their
  endomorphisms to characteristic zero}, Algebraic geometry, {B}owdoin, 1985
  ({B}runswick, {M}aine, 1985), Proc. Sympos. Pure Math., vol.~46, Amer. Math.
  Soc., Providence, RI, 1987, pp.~165--195. \MR{927980}

\bibitem{orlov97}
D.~O. Orlov, \emph{Equivalences of derived categories and {$K3$} surfaces}, J.
  Math. Sci. (New York) \textbf{84} (1997), no.~5, 1361--1381, Algebraic
  geometry, 7. \MR{1465519}

\bibitem{orlov02}
\bysame, \emph{Derived categories of coherent sheaves on abelian varieties and
  equivalences between them}, Izv. Ross. Akad. Nauk Ser. Mat. \textbf{66}
  (2002), no.~3, 131--158. \MR{1921811}

\bibitem{ploog07}
David Ploog, \emph{Equivariant autoequivalences for finite group actions}, Adv.
  Math. \textbf{216} (2007), no.~1, 62--74. \MR{2353249}

\bibitem{Polishchuk2003}
Alexander Polishchuk, \emph{Abelian varieties, theta functions and the
  {F}ourier transform}, Cambridge Tracts in Mathematics, vol. 153, Cambridge
  University Press, Cambridge, 2003. \MR{1987784}

\bibitem{romagny05}
Matthieu Romagny, \emph{Group actions on stacks and applications}, Michigan
  Math. J. \textbf{53} (2005), no.~1, 209--236. \MR{2125542}

\bibitem{Shioda1977}
Tetsuji Shioda, \emph{Some remarks on {A}belian varieties}, J. Fac. Sci. Univ.
  Tokyo Sect. IA Math. \textbf{24} (1977), no.~1, 11--21. \MR{450289}

\bibitem{shioda78}
\bysame, \emph{Supersingular {$K3$} surfaces}, Algebraic geometry ({P}roc.
  {S}ummer {M}eeting, {U}niv. {C}openhagen, {C}openhagen, 1978), Lecture Notes
  in Math., vol. 732, Springer, Berlin, 1979, pp.~564--591. \MR{555718}

\bibitem{stacks-project}
The {Stacks project authors}, \emph{The stacks project},
  \url{https://stacks.math.columbia.edu}.

\bibitem{Stellari2007}
Paolo Stellari, \emph{Derived categories and {K}ummer varieties}, Math. Z.
  \textbf{256} (2007), no.~2, 425--441. \MR{2289881}

\bibitem{vistolinote}
Angelo Vistoli, \emph{Grothendieck topologies, fibered categories and descent
  theory}, Fundamental algebraic geometry, Math. Surveys Monogr., vol. 123,
  Amer. Math. Soc., Providence, RI, 2005, pp.~1--104. \MR{2223406}

\end{thebibliography}

\end{document}